\documentclass[11pt,reqno]{amsart}
\usepackage{amssymb}
\usepackage{verbatim}
\usepackage{tikz}
\usepackage{xcolor}
\usepackage{graphicx}

\newcommand{\NN}{\mathcal N}

\newcommand{\be}{\begin{equation}}
\newcommand{\ee}{\end{equation}}
\newtheorem{lemma}{Lemma}
\newtheorem{definition}{Definition}
\newtheorem{theorem}{Theorem}

\newtheorem{proposition}{Proposition}

\newtheorem{remark}{Remark}

\begin{document}

\baselineskip 16pt

\title{On star-forest ascending subgraph decomposition}

\date{}

\author{J.M. Aroca, A. Llad\'o}
\address{Univ. Polit\`ecnica de Catalunya}
\email{aina.llado@upc.edu}

\maketitle

\begin{abstract} 
The    Ascending Subgraph Decomposition (ASD) Conjecture
asserts that every graph $G$ with ${n+1\choose 2}$ edges admits an
edge decomposition $G=H_1\oplus\cdots \oplus H_n$ such that $H_i$
has $i$ edges and it is isomorphic to a subgraph of $H_{i+1}$,
$i=1,\ldots ,n-1$. We show that every bipartite graph $G$ with
${n+1\choose 2}$ edges such that the degree sequence $d_1,\ldots
,d_k$ of one of the stable sets satisfies $ d_{k-i}\ge n-i\; \text{for each}\; 0\le i\le k-1,$, admits an ascending subgraph decomposition with star forests.
We also give a necessary condition on the degree sequence which is
not far from the above sufficient one.
\end{abstract}

\section{Introduction}

A graph $G$ with ${n+1\choose 2}$ edges has an Ascending Subgraph
Decomposition (ASD) if it admits an edge--decomposition
$G=G_1\oplus\cdots \oplus G_n$ such that   $G_i$ has $i$ edges and it 
is isomorphic to a subgraph of $G_{i+1}$, $1\le i<n$. Throughout this paper we use the
symbol $\oplus$ to denote edge--disjoint union of graphs. It was
conjectured by Alavi, Boals, Chartrand, Erd\H{o}s and  Oellerman
\cite{alavi} that every graph  of size ${n+1\choose 2}$ admits
an ASD. The conjecture has been proved for a number of particular
cases, including forests \cite{faudree}, regular graphs
\cite{fuhu3}, complete multipartite graphs \cite{fuhu2} or graphs
with maximum degree $\Delta \le n/2$ \cite{fu}.

In the same paper Alavi et al. \cite{alavi} conjectured that every
star forest of size ${n+1\choose 2}$ in which each connected
component has size at least $n$ admits an ASD in which every graph
in the decomposition is a star. This conjecture was proved by Ma,
Zhou and Zhou \cite{ma}, and the condition was later on weakened to
the effect that the second smaller component of the star forest has
size at least $n$ by Chen, Fu , Wang and Zhou \cite{chen}.

The class of bipartite graphs which admits a star--forest--ASD is clearly larger than the one which admit star --ASD, see Figure \ref{fig1} for a simple example. This motivates the  study of star--forest--ASD for bipartite graphs in terms of the degree sequence of one of the stable sets, which is the purpose of this paper.

\begin{figure}[h]
\begin{center}
\begin{tikzpicture}[scale=0.6]
\foreach \i in {1,2,3,4}
{
\path (0,\i-1) coordinate (P\i);
\path (2,\i-0.5) coordinate (Q\i);
}
\foreach \i in {1,2,3}
{
\draw[fill] (P\i) circle (2pt);
\draw[fill] (Q\i) circle (2pt);
}
\draw[fill] (P4) circle (2pt);
\draw[gray] (P1)--(Q1) ;
\draw[gray] (P2)--(Q2) (P3)--(Q3);
\draw[gray] (P4)--(Q3) (P3)--(Q2) (P3)--(Q1);
\node at (2.8,1.5) {$=$};
\end{tikzpicture}
\hspace{5mm}
\begin{tikzpicture}[scale=0.6]
\foreach \i in {1,2,3,4}
{
\path (0,\i-1) coordinate (P\i);
\path (2,\i-0.5) coordinate (Q\i);
}
\foreach \i in {1,2,3}
{
\draw[fill] (P\i) circle (2pt);
\draw[fill] (Q\i) circle (2pt);
}
\draw[fill] (P4) circle (2pt);
\draw[thick] (P1)--(Q1) ;
\draw[lightgray] (P2)--(Q2) (P3)--(Q3);
\draw[lightgray] (P4)--(Q3) (P3)--(Q2) (P3)--(Q1);
\node at (2.8,1.5) {$\oplus$};
\end{tikzpicture}
\begin{tikzpicture}[scale=0.6]
\foreach \i in {1,2,3,4}
{
\path (0,\i-1) coordinate (P\i);
\path (2,\i-0.5) coordinate (Q\i);
}
\foreach \i in {1,2,3}
{
\draw[fill] (P\i) circle (2pt);
\draw[fill] (Q\i) circle (2pt);
}
\draw[fill] (P4) circle (2pt);
\draw[lightgray] (P1)--(Q1) ;
\draw[thick] (P2)--(Q2) (P3)--(Q3);
\draw[lightgray] (P4)--(Q3) (P3)--(Q2) (P3)--(Q1);
\node at (2.8,1.5) {$\oplus$};
\end{tikzpicture}
\begin{tikzpicture}[scale=0.6]
\foreach \i in {1,2,3,4}
{
\path (0,\i-1) coordinate (P\i);
\path (2,\i-0.5) coordinate (Q\i);
}
\foreach \i in {1,2,3}
{
\draw[fill] (P\i) circle (2pt);
\draw[fill] (Q\i) circle (2pt);
}
\draw[fill] (P4) circle (2pt);
\draw[lightgray] (P1)--(Q1) ;
\draw[lightgray] (P2)--(Q2) (P3)--(Q3);
\draw[thick] (P4)--(Q3) (P3)--(Q2) (P3)--(Q1);
\end{tikzpicture}
\end{center}
\caption{Bipartite graph with Star--Forest--ASD but with no Star--ASD.}
\label{fig1}
\end{figure}
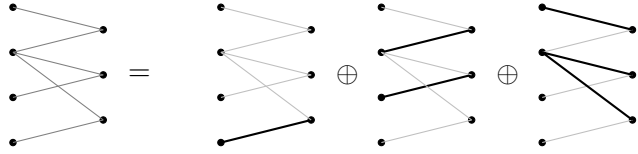

Faudree, Gy\'arf\'as and Schelp \cite{faudree} proved that every forest of stars admits a star--forest--ASD. These authors mention, in the same paper,  that ``Surprisingly [this result], is the most difficult to prove. This could indicate that the conjecture (if true) is a difficult one to prove".  In the same paper the authors propose the following question: Let $G$ be a graph with ${n+1\choose 2}$ edges. Does $G$ have an ASD such that each member is a star forest? In this paper we address this question for bipartite graphs when the centres of the stars in the star--forest--ASD belong to the same stable set. Our main result is the following one.

\begin{theorem}\label{thm:starforest}
Let $G$ be a bipartite graph with
${n+1\choose 2}$ edges. Let $ d_1\le d_2\le\cdots \le d_k$ be the degree sequence of one of the stable sets of $G$.
 If
 $$
 d_{k-i}\ge n-i\; \text{for each}\; 0\le i\le k-1,
 $$
 then, there is a star--forest--ASD of $G$.
\end{theorem}

The proof of Theorem \ref{thm:starforest} is made in two steps. First we prove   the result for a class of bipartite graphs which we call reduced graphs. For this we use a representation of a star-forest decomposition by the so--called ascending matrices and certain multigraphs and reduce the problem to the existence of a particular edge--coloring of these multigraphs.  The terminology and the proof for reduced graphs is contained in Section \ref{sec2}. In Section  \ref{sec3} we present an extension lemma, which uses a result of H\"aggkvist \cite{hagg} on list edge--colorings, which allows one to extend the decomposition from reduced  to all bipartite graphs with the same degree sequence on one of the stable sets,  completing the proof of Theorem \ref{thm:starforest}. The final section contains   some concluding remarks.

The sufficient condition on the degrees given in Theorem \ref{thm:starforest} is not far from being necessary.

\begin{lemma}\label{lem:starforest} If every bipartite graph $G$ with degree sequence  $ d_1\le d_2\le
\cdots \le d_k$  of  one stable sets $X$  of $G$  admits a star--forest decomposition (ascending or not) with the centers of the stars in $X$ then
$$
\sum_{i=0}^t d_{k-i}\ge \sum_{i=0}^t (n-i),\; t=0,\ldots ,k-1.
$$
\end{lemma}


\section{Star-forest--ASD for reduced bipartite graphs}
\label{sec2}

Thorughout the section  $G=G(X,Y)$ denotes  a bipartite graph with color classes  $X=\{x_1,\ldots ,x_k\}$ and $Y=\{y_1,\ldots ,y_n\}$. We
denote by $d_X=(d_1\le \cdots \le d_k)$, $d_i=d(x_i)$,  the degree sequence of the
vertices in the stable set $X$ of $G$. We call $d_X$ the $X$--{\it degree sequence} of $G$. We focus on star forest ASD with the stars of the decomposition centered at the vertices in $X$.

We first introduce some definitions.

\begin{definition}[Reduced graph] The reduced graph $G_R=G_R(X,Y')$ of $G(X,Y)$  has color classes $X$ and $Y'=\{y'_1,\ldots ,y'_{d_k}\}$ and $x_i$ is adjacent to the vertices $y'_1,\ldots ,y'_{d_i}$, $i=1,\ldots ,k$. 

We say that $G$ is reduced if $G=G_R$.
\end{definition}

Figure \ref{fig-red} illustartes the above definition.

\begin{figure}[h]
\begin{tikzpicture}[scale=0.6]
\foreach \i in {1,...,4}
{
\path (0,\i+1)  coordinate  (P\i);
\draw[fill] (P\i)  circle (2pt);
\node[left] at (P\i) {$x_{\i}$};
}
\foreach \i in {1,...,6}
{
\path (2,\i)  coordinate  (Q\i);
\draw[fill] (Q\i)  circle (2pt);
\node[right] at (Q\i) {$y_{\i}$};
}
\draw (P4)--(Q6) (P4)--(Q5) (P4)--(Q4) (P4)--(Q3) (P4)--(Q2);
\draw (P3)--(Q5) (P3)--(Q3);
\draw (P2)--(Q3) (P2)--(Q1);
\draw (P1)--(Q1);
\node at (4,3.5) {$\longrightarrow$};
\node at (1.5,0) {$G$};
\end{tikzpicture}
\hspace{4mm}
\begin{tikzpicture}[scale=0.6]
\foreach \i in {1,...,4}
{
\path (0,\i+0.5)  coordinate  (P\i);
\draw[fill] (P\i)  circle (2pt);
\node[left] at (P\i) {$x_{\i}$};
}
\foreach \i in {1,...,5}
{
\path (2,\i)  coordinate  (Q\i);
\draw[fill] (Q\i)  circle (2pt);
\node[right] at (Q\i) {$y_{\i}'$};
}
\draw (P4)--(Q4) (P4)--(Q3) (P4)--(Q2) (P4)--(Q1) (P4)--(Q5);
\draw (P3)--(Q1) (P3)--(Q2);
\draw (P2)--(Q1) (P2)--(Q2);
\draw (P1)--(Q1);
\node at (1.5,0) {$G_{R}$};
\end{tikzpicture}
\caption{A bipartite graph and its reduced graph.}
\end{figure}
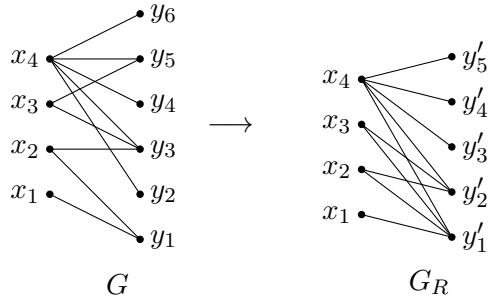\label{fig-red}

Given two $k$--dimensional vectors $c=(c_1,\ldots ,c_k)$ and
$c'=(c'_1,\ldots ,c'_k)$, we say that $c\preceq c'$ if after reordering the components
of each vector in nondecreasing order, the $i$--th component of $c$
is not larger than the $i$--th component of $c'$. This definition is
motivated by the following remark.

\begin{remark}
\label{rem:subforest} 
Let $F, F'$ be two forests of
stars with centers $x_1,\ldots ,x_k$ and $x'_1,\ldots ,x'_k$
respectively. Then $F$ is isomorphic to a subgraph of $F'$ if and
only if $(d_F(x_1),\ldots ,d_F(x_k))\preceq (d_{F'}(x_1'),\ldots
,d_{F'}(x_k'))$.
\end{remark}

Given two sequences $d=(d_1\le \cdots  \le d_k)$ and $b=(b_1\le \cdots \le b_n)$ of nonnegative integers with $\sum_id_i=\sum_jb_j$,   denote by $\NN(d,b)$ the set of $k\times n$ matrices $A$ with nonnegative integer entries such that the 
row sums  satisfy $\sum_j a_{ij} =d_i$, $i=1,\ldots ,k$ and the column sums  satisfy  $\sum_i a_{ij} =b_j$, $j=1,\ldots ,n$.

\begin{definition}[Ascending matrix]  
We say that a matrix  $A\in\NN(d,b)$ if  in addition to the row sum being the  sequence $d$ and the  column sum  the sequence $b$, we have
$$
A_1\preceq A_2\preceq \cdots \preceq A_n,
$$
where $A_j$ denotes the $j$--th column of $A$.
\end{definition}

For convenience we use the following notation for sequences. The constant sequence with $r$ entries equal to $x$ is denoted by $x^r$ and $(x^r, y^s)$ denotes the concatenation of $x^r$ and $y^s$. Also, for an integer $x$ we denote by $x^-$ the ascending sequence $x^-=(1,2,\ldots ,x-1,x)$. Sums and differences of sequences of the same length are understood to be componentwise.

We will use appropriate ascending matrices to define multigraphs which will lead to starforest-ASD as stated in Proposition \ref{prop:goodascending} below. We recall that the bipartite adjacency matrix of a bipartite multigraph $H$ with color classes $X=\{x_1,\ldots ,x_k\}$ and $Z=\{ z_1,\ldots ,z_n\}$ is the $(k\times n)$ matrix $A$ where $a_{ij}$ is the number of edges joining $x_i$ with $z_j$.

We need a last definition which is borrowed from \cite{hagg}.

\begin{definition}[Sequential coloring]
A bipartite multigraph $H$ with degree sequence $d_X=(d_1\le\cdots \le d_k)$ on the color class $X$ has a {\it sequential coloring} for $X$ if there is a proper edge coloring of $H$ such that the edges incident with vertex $x_i$ receive colors $\{1,\ldots ,d_i\}$ for each $i$.
\end{definition}

\begin{proposition}
\label{prop:goodascending} 
A reduced bipartite graph $G=G(X,Y)$ with degree--sequence $d_X=(d_1\le \cdots \le d_k)$ has a star--forest--ASD with centers of stars in $X$ if and only if there is an $A\in\NN(d_X,n^-)$ such that the bipartite multigraph $H=H(X,Z)$ with bipartite incidence matrix  $A$ admits a sequential coloring.
\end{proposition}

\begin{proof} Let $X=\{x_1,\ldots ,x_k\}$ and $Y=\{ y_1,\ldots ,y_{d_k}\}$ denote the sets in the bipartition of $G$.

Assume first that   $G=G(X,Y)$ is a  reduced bipartite graph with  $d_X=(d_1\le \cdots \le d_k)$ which admits a star--forest--ASD
$$
 G=F_1\oplus\cdots \oplus F_n
 $$ 
 with the centers of the stars in $X$. 

We  define the multigraph $H=H(X,Z)$ with $Z=\{ z_1,\ldots ,z_n\}$ by  placing  $d_{F_j}(x_i)$ parallel edges joining $x_i$ with $z_j$, where $d_{F_j}(x_i)$ denotes the degree of $x_i$ in the forest $F_j$. In this way, the bipartite adjacency matrix $A$ of $H$ is a $(d_X,n^-)$--matrix. Moreover, since $F_j$ is isomorphic to a subgraph of $F_{j+1}$, the matrix $A$ is $(d,n^-)$--ascending. 

Next we define an edge--coloring of $H$ as follows.  Denote by $N_{F_j}(x_i)=\{y_{i_1},\ldots ,y_{i_s}\}$ the set of neighbours of $x_i$ in the forest $F_j$. Then $x_i$ is joined in $H$ to $z_j$ with  $s$ parallel edges. We color these edges with the subscripts $\{i_1,\ldots ,i_s\}$ of the neighbours of $x_i$ in $F_j$, by assigning one of these colors to each parallel edge bijectively. In other words, if we define $I_{ij}\subset \{1,\ldots ,d_k\}$ by 
$$
I_{ij}=\{h\in \{1,\ldots ,d_k\}: x_iy_h\in E(F_j)\},
$$
then the coloring is defined by any bijection from the $|I_{ij}|$ parallel edges joining $x_i$ with $z_j$ in $H$ to $I_{ij}$. 

Since the original graph $G$ is simple and the star--forests $F_1,\ldots ,F_n$ form a decomposition of $G$, no two edges incident to a vertex $x_i$ receive the same color. On the other hand, since each star--forest $F_j$ has its stars centered in vertices in $X$, and therefore each vertex in $Y$ has degree at most one in each $F_j$, by the bijections which define the coloring, no two edges incident to $z_j$ receive the same color. Hence, the coloring is proper. Moreover, since the graph $G$ is reduced, the edges incident to $x_i$ receive the colors $\{1,\ldots ,d_i\}$ and the coloring is sequential. This completes the if part of the proof.

Reciprocally, assume  that $A\in\NN(d_X,n^-)$ and that  the multigraph $H=H(X,Z)$ with bipartite adjacency matrix $A$ has a sequential coloring.

Let $Z=\{z_1,\ldots ,z_n\}$, where vertex $z_i$ has degree $i$ (the sum of entries of column $i$ of $A$), in $H$,$1\le i\le n$. Let $c:E\to \{1,2,\ldots ,d_k\}$ be a sequential coloring of $H$, so that the edges incident to $x_i$ receive colors $\{1,\ldots ,d_i\}$.   

Each $z_j$ will be associated to the subgraph $F_j$ of $G$ defined as follows. For each edge $x_iz_j$ of color $h$ we declare the edge $x_iy_h$ to be in $F_j$. This way we obtain a subgraph of $G$ because $h\le d(x_i)$ (the coloring is sequential) and the graph $G$ is reduced. Moreover, since the coloring is proper, the degree of every vertex $y_h$ in $F_j$ is at most one. Hence $F_j$ is a forest of stars and it has $j$ edges. Moreover, for $j\neq j'$, the subgraphs  $F_j$ and $F_{j'}$ are edge--disjoint again from the fact that the coloring is proper. Finally, since the matrix $A$ is ascending, $F_i$ is isomorphic to a subgraph of $F_{i+1}$ for each $i=1,\cdots,n-1$. Hence,
$$
G=F_1\oplus\cdots \oplus F_n
$$
is a starforest--ASD for $G$. This completes the proof.
\end{proof}

Figure \ref{figura} illustrates the statement of Proposition \ref{prop:goodascending}. 
\begin{figure}[h]
\begin{tikzpicture}[scale=0.6]
\node at (1,7) {$G$};
\foreach \i in {1,...,5}
{
\path (0,6.5-\i) coordinate (P\i);
\draw[fill] (P\i) circle (2pt);
\node[left] at (P\i) {$x_{\i}$};
}
\foreach \i in {1,...,6}
{
\path (2,7-\i) coordinate (Q\i);
\draw[fill] (Q\i) circle (2pt);
\node[right] at (Q\i) {$y_{\i}$};
}
\draw[thin,gray] (P5)--(Q6) (P5)--(Q5) (P5)--(Q4) (P5)--(Q3) (P5)--(Q2) (P5)--(Q1);
\draw[thin,gray] (P4)--(Q3) (P4)--(Q2) (P4)--(Q1);
\draw[thin,gray] (P3)--(Q3) (P3)--(Q2) (P3)--(Q1);
\draw[thin,gray] (P2)--(Q2) (P2)--(Q1);
\draw[thin,gray] (P1)--(Q1) ;
\node at (3.5,3.5) {$=$};
\end{tikzpicture}
\begin{tikzpicture}[scale=0.6]
\node at (1,7) {$F_1$};
\foreach \i in {1,...,5}
{
\path (0,6.5-\i) coordinate (P\i);
\draw[fill] (P\i) circle (2pt);
}
\foreach \i in {1,...,6}
{
\path (2,7-\i) coordinate (Q\i);
\draw[fill] (Q\i) circle (2pt);
}
\draw[thin,gray] (P4)--(Q1) ;
\node at (3.5,3.5) {$\oplus$};
\end{tikzpicture}
\begin{tikzpicture}[scale=0.6]
\node at (1,7) {$F_2$};
\foreach \i in {1,...,5}
{
\path (0,6.5-\i) coordinate (P\i);
\draw[fill] (P\i) circle (2pt);
}
\foreach \i in {1,...,6}
{
\path (2,7-\i) coordinate (Q\i);
\draw[fill] (Q\i) circle (2pt);
}
\draw[thin,gray]  (P5)--(Q1);
\draw[thin,gray] (P3)--(Q2);
 \node at (3.5,3.5) {$\oplus$};
\end{tikzpicture}
\begin{tikzpicture}[scale=0.6]
\node at (1,7) {$F_3$};
\foreach \i in {1,...,5}
{
\path (0,6.5-\i) coordinate (P\i);
\draw[fill] (P\i) circle (2pt);
}
\foreach \i in {1,...,6}
{
\path (2,7-\i) coordinate (Q\i);
\draw[fill] (Q\i) circle (2pt);
}
\draw[thin,gray]  (P5)--(Q2);
\draw[thin,gray] (P4)--(Q3);
\draw[thin,gray] (P2)--(Q1);
\node at (3.5,3.5) {$\oplus$};
\end{tikzpicture}
\begin{tikzpicture}[scale=0.6]
\node at (1,7) {$F_4$};
\foreach \i in {1,...,5}
{
\path (0,6.5-\i) coordinate (P\i);
\draw[fill] (P\i) circle (2pt);
}
\foreach \i in {1,...,6}
{
\path (2,7-\i) coordinate (Q\i);
\draw[fill] (Q\i) circle (2pt);
}
\draw[thin,gray]  (P5)--(Q4);
\draw[thin,gray]  (P5)--(Q3);
\draw[thin,gray] (P3)--(Q1);
\draw[thin,gray] (P2)--(Q2);

\node at (3.5,3.5) {$\oplus$};
\end{tikzpicture}
\begin{tikzpicture}[scale=0.6]
\node at (1,7) {$F_5$};

\foreach \i in {1,...,5}
{
\path (0,6.5-\i) coordinate (P\i);
\draw[fill] (P\i) circle (2pt);
}
\foreach \i in {1,...,6}
{
\path (2,7-\i) coordinate (Q\i);
\draw[fill] (Q\i) circle (2pt);
}
\node[white] at (Q6) {$y_1$};
\draw[thin,gray] (P5)--(Q5) (P5)--(Q6);
\draw[thin,gray] (P4)--(Q2);
\draw[thin,gray]  (P3)--(Q3);
\draw[thin,gray] (P1)--(Q1);
\end{tikzpicture}
 
\begin{center}
\begin{tikzpicture}[scale=0.6]

\node at (-9,2) {$A=
\left(\begin{array}{ccccc}
0 & 0 & 0 & 0 & 1 \\
0 & 0 & 1 & 1 & 0 \\
0 & 1 & 0 & 1 & 1 \\
1 & 0 & 1 & 0 & 1 \\
0 & 1 & 1 & 2 & 2
\end{array}\right)
$};

\foreach \i in {1,...,5}
{
\path (0,5-\i) coordinate (P\i);
\draw[fill] (P\i) circle (2pt);
\path (2,5-\i) coordinate (Q\i);
\draw[fill] (Q\i) circle (2pt);
\node[left] at (P\i) {$x_{\i}$};
\node[right] at (Q\i) {$z_{\i}$};
}
\draw[thin,gray] (P5) to [out=30,in=150] (Q5) ;
\draw[thin,gray] (P5) to [out=-30,in=-150] (Q5) ;
\draw[thin,gray] (P5) to [out=40,in=180] (Q4) ;
\draw[thin,gray] (P5) to [out=0,in=-120] (Q4) ;
\draw[thin,gray]  (P5)--(Q3) (P5)--(Q2) ;
\draw[thin,gray] (P4)--(Q5) (P4)--(Q3) (P4)--(Q1);
\draw[thin,gray] (P3)--(Q5) (P3)--(Q4)  (P3)--(Q2) ;
\draw[thin,gray] (P2)--(Q4) (P2)--(Q3);
\draw[thin,gray] (P1)--(Q5);
\node at (-2,2) {$H=$};
\end{tikzpicture}
\end{center}

\begin{tikzpicture}[scale=0.6]
\node at (1,5) {$y_6$};
\foreach \i in {1,...,5}
{
\path (0,5-\i) coordinate (P\i);
\draw[fill] (P\i) circle (2pt);
\path (2,5-\i) coordinate (Q\i);
\draw[fill] (Q\i) circle (2pt);
}
\draw[thin,gray] (P5) to [out=30,in=150] (Q5) ;
\draw[thin,white] (P5) to [out=-30,in=-150] (Q5) ;
\node at (3,2) {$\oplus$};
\node at (-2,2) {$H=$};
\end{tikzpicture}
\begin{tikzpicture}[scale=0.6]
\node at (1,5) {$y_5$};
\foreach \i in {1,...,5}
{
\path (0,5-\i) coordinate (P\i);
\draw[fill] (P\i) circle (2pt);
\path (2,5-\i) coordinate (Q\i);
\draw[fill] (Q\i) circle (2pt);
}
 \draw[thin,gray] (P5) to [out=-30,in=-150] (Q5) ;
\node at (3,2) {$\oplus$};
\end{tikzpicture}
\begin{tikzpicture}[scale=0.6]
\node at (1,5) {$y_4$};
\foreach \i in {1,...,5}
{
\path (0,5-\i) coordinate (P\i);
\draw[fill] (P\i) circle (2pt);
\path (2,5-\i) coordinate (Q\i);
\draw[fill] (Q\i) circle (2pt);
} 
\draw[thin,gray] (P5) to [out=0,in=-120] (Q4) ;
\draw[thin,white] (P5) to [out=-30,in=-150] (Q5) ;
\node at (3,2) {$\oplus$};
\end{tikzpicture}
\begin{tikzpicture}[scale=0.6]
\node at (1,5) {$y_3$};
\foreach \i in {1,...,5}
{
\path (0,5-\i) coordinate (P\i);
\draw[fill] (P\i) circle (2pt);
\path (2,5-\i) coordinate (Q\i);
\draw[fill] (Q\i) circle (2pt);
}
\draw[thin,gray] (P5) to [out=40,in=180] (Q4) ;
\draw[thin,gray]   (P4)--(Q3) ;
\draw[thin,gray] (P3)--(Q5) ;
\draw[thin,white] (P5) to [out=-30,in=-150] (Q5) ;
\node at (3,2) {$\oplus$};
\end{tikzpicture}
\begin{tikzpicture}[scale=0.6]
\node at (1,5) {$y_2$};
\foreach \i in {1,...,5}
{
\path (0,5-\i) coordinate (P\i);
\draw[fill] (P\i) circle (2pt);
\path (2,5-\i) coordinate (Q\i);
\draw[fill] (Q\i) circle (2pt);
}
\draw[thin,gray]   (P5)--(Q3);
\draw[thin,gray] (P4)--(Q5);
\draw[thin,gray]   (P3)--(Q2);
\draw[thin,gray] (P2)--(Q4) ;
 \draw[thin,white] (P5) to [out=-30,in=-150] (Q5) ;
\node at (3,2) {$\oplus$};
\end{tikzpicture}
\begin{tikzpicture}[scale=0.6]
\node at (1,5) {$y_1$};
\foreach \i in {1,...,5}
{
\path (0,5-\i) coordinate (P\i);
\draw[fill] (P\i) circle (2pt);
\path (2,5-\i) coordinate (Q\i);
\draw[fill] (Q\i) circle (2pt);
}
\draw[thin,gray]   (P5)--(Q2);
\draw[thin,gray] (P4)--(Q1);
\draw[thin,gray] (P3)--(Q4) ;
\draw[thin,gray] (P2)--(Q3);
\draw[thin,gray] (P1)--(Q5);
 \draw[thin,white] (P5) to [out=-30,in=-150] (Q5) ;
\end{tikzpicture}
\caption{Illustration of Proposition  \ref{prop:goodascending} with $d_X=(1,2,3,3,6)$ and $n=5$.}
\label{figura}
\end{figure}
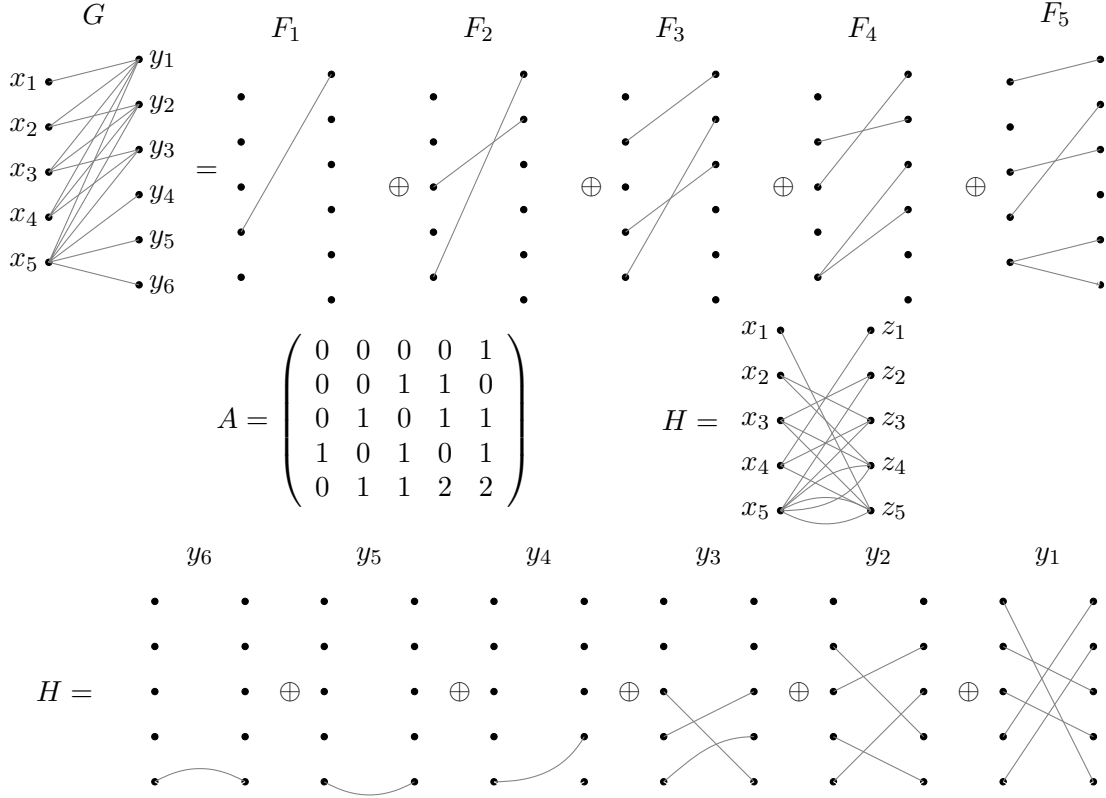

In order to prove the main result for reduced graphs we will show the existence of an appropriate ascending matrix and that the multigraph associated to it admits a sequential coloring.

According to Remark \ref{rem:subforest},  the above mentioned result \cite{faudree} on the existence of a star--forest--ASD of every star forest with ${n+1\choose 2}$ edges can be reformulated as the existence of an ascending matrix $A\in\NN(d,n^-)$ for every sequence $d=(d_1\le\cdots \le d_k)$ with $\sum_{i=1}^kd_i={n+1\choose 2}$:

\begin{lemma}\label{lemasc}
For every sequence  $d=(d_1\le \cdots \le d_k)$  there is a $(d,n^-)$--ascending matrix $C$.
\end{lemma}

We next show that there exists an ascending matrix  $A\in\NN(d, n^-)$  of a particular shape that will be useful to prove the existence of star--forest--ASD for reduced graphs with that degree sequence.

We note that, if $b_1\le\cdots \le b_n$, then each matrix  $A\in\NN(a,b)$ with $(0,1)$ entries is ascending. The support of a matrix $B$ is the set of positions with nonzero entries. We observe that if  $B\in\NN(a,b)$, $T\in\NN(a',b')$ and $B$ and $T$ have disjoint support and the same dimensions,  then $B+T\in\NN(a+a',b+b')$. The last sentence also holds if the support of $T$ and $B$ intersect in a square submatrix and $T$ has constant entries in this submatrix. The above observations will be used in the proof of the next Lemma, which is illustrated in  Figure \ref{fig:lasc}.

\begin{lemma}\label{lasc}
Let $d=(d_1\le \cdots \le d_k)$ be a sequence satisfying $\sum_{d_i}={n+1\choose 2}$  and
$$
d_{k-i}\ge n-i,\; i=0,\ldots ,k-1.
$$
There is a matrix  $A\in\NN(d,n^-)$ such that $a_{ij}\ge 1$ for each $(i,j)$ with $i+j\ge k+1$.
\end{lemma}

\begin{proof} 
Consider the $(k\times n)$ matrix $T$ with $t_{ij}=1$ for $i+j\ge k+1$ and $t_{ij}=0$ otherwise. Let $d'_k=d_k-n, d'_{k-1}=d_{k-1}-(n-1),\ldots ,d'_1=d_1-(n-k)$. Since $\sum_id'_i=(n-k)+(n-k-1)+\cdots +2+1$, by Lemma \ref{lemasc}  for $d'=(d'_1,\ldots ,d'_k)$ there is a $(k\times (n-k))$ matrix $A\in\NN(d',(n-k)^-)$. Extend  this matrix to a $k\times n$ matrix $A'$ by adding zero columns to the right. Since the last $(n-k)$ columns of $T$ are the all--ones vectors, the matrix $A=A'+T$ still has the ascending column property and, by construction, it is in  $\NN(d,n^-)$ with nonzero entries in the positions $(i,j)$ with $i+j\ge k+1$.
\end{proof}

\begin{figure}[h]
$$
T=\left(\begin{array}{ccccccc}
0 & 0 & 0 & 1& 1 & 1 & 1 \\
0 & 0 & 1 & 1 & 1 & 1 & 1 \\
0 & 1 & 1 & 1 & 1 & 1 & 1 \\
1 & 1 & 1 & 1 & 1 & 1 & 1
\end{array}\right);\;\;
A'=\left(\begin{array}{ccccccc}
0 & 0 & 0 & 0 & 0 & 0 & 0 \\
0 & 0 & 0 & 0 & 1 & 0 & 0 \\
0 & 0 & 0 & 0 & 0 & 1 & 2\\
0 & 0 & 0 & 0 & 0 & 1 & 1
\end{array}\right)
$$

$$
A=A'+T=
\left(\begin{array}{ccccccc}
0 & 0 & 0 & 1& 1 & 1 & 1 \\
0 & 0 & 1 & 1 & 2 & 1 & 1 \\
0 & 1 & 1 & 1 & 1 & 2 & 3 \\
1 & 1 & 1 & 1 & 1 & 2 & 2
\end{array}\right)
$$
\caption{An illustration of Lemma \ref{lasc} with $n=7$ and $d=(4,6,9,9)$.}
\end{figure}\label{fig:lasc}

Next Lemma gives a sufficient condition for a degree sequence of a reduced graph to admit a star-forest--ASD.

\begin{lemma}\label{lemreduc}
Let  $d=(d_1\le \cdots \le d_k)$ be a sequence of positive integers with $\sum_id_i={n+1\choose 2}$.  If
 $$
 d_{k-i}\ge n-i,\; \;i=0,1,\ldots ,k-1,
 $$ then the reduced graph with degree sequence $d$ admits a star--forest ASD.
\end{lemma}

\begin{proof} 
Let $A\in\NN(d,n^-)$ such that $a_{ij}\ge 1$ for each $(i,j)$ with $i+j\ge k+1$,  whose existence is ensured by Lemma \ref{lasc}.

Let $H$ be the bipartite multigraph with stable sets $X=\{x_1,\ldots ,x_k\}$ and $Z=\{z_1,\ldots ,z_n\}$ whose bipartite adjacency matrix is $A$.  We next show that $H$ admits a sequential coloring. The result will follow by Proposition \ref{prop:goodascending}.

Let $\alpha_i=d_i-(n-k)$, $1\le i\le k$. For each $i=1,\dots ,k$ denote by $M_i'$ the matching in $H$  formed by the $k$ edges
$$
M'_i=\{ x_{r}z_{s}: r+s\equiv\{i, n-k+i\} \pmod{n}\}.
$$
Such matchings exist in $H$  by the condition $a_{ij}\ge 1$ for each pair
$(i,j)$ with $i+j\le k+1$, and they are pairwise edge--disjoint. For each $j=0,1,\ldots ,k-1$, let $M_j\subset M'_j$ be obtained by selecting from $M'_j$ the edges incident to $x_i$ whenever $\alpha_i\ge j$.  In this way each $x_i$ is incident with the matchings $M_1,\ldots ,M_{t(i)}$ with $t(i)=\min\{k,\alpha_i\}$ and has degree at least $n-k$ in $M_1\oplus \cdots\oplus M_k$. On the other hand, by the condition on the degree sequence $d_X$, since $\alpha_i\ge k-i+1$, the vertex  $z_{n-i}$ is incident to $k-i$ edges in  $M_1\oplus\cdots \oplus M_k$.

Let $H'$ denote the bipartite multigraph obtained from $H$ by removing the edges in $M_1\oplus\cdots\oplus M_k$. Let $d'_X=(d'_1\le \cdots \le d'_k)$ be the
degree sequence of   $X$  in $H'$, where $d'_{i}=d_{i}-t(i)\ge n-k$. Moreover, each vertex $z_i$, $1\le i\le n$  has degree at most $n-k$ in $H'$ (see an example in Figure \ref{fig:lemreduc}.)

\begin{figure}[h]
\begin{tikzpicture}
\node at (0,1.2) {$H$};
\node at (0,0) {$\left(\begin{array}{ccccccc}
0 & 0 & 0 & 1& 1 & 1 & 1 \\
0 & 0 & 1 & 1 & 2 & 1 & 1 \\
0 & 1 & 1 & 1 & 1 & 2 & 3 \\
1 & 1 & 1 & 1 & 1 & 2 & 2
\end{array}\right)$};
\node at (2.5,0) {$=$};
\node at (5,1.2) {$M'_1$};
\node at (5,0) {$\left(\begin{array}{ccccccc}
0 & 0 & 0 & 0& 0 & 0 & 1 \\
0 & 0 & 0 & 0 & 0 & 1 & 0 \\
0 & 0 & 0 & 0 & 1 & 0 & 0 \\
0 & 0 & 0 & 1 & 0 & 0 & 0
\end{array}\right)$};
\node at (7.5,0) {$+$};
\node at (10,1.2) {$M'_2$};
\node at (10,0) {$\left(\begin{array}{ccccccc}
0 & 0 & 0 & 0& 0 & 0 & 0 \\
0 & 0 & 0 & 0 & 0 & 0 & 1 \\
0 & 0 & 0 & 0 & 0 & 1 & 0 \\
0 & 0 & 0 & 0 & 1 & 0 & 0
\end{array}\right)$};

\node at (0,-1.8) {$M'_3$};
\node at (0,-3) {$\left(\begin{array}{ccccccc}
0 & 0 & 0 & 0& 0 & 0 & 0 \\
0 & 0 & 0 & 1 & 0 & 0 & 0 \\
0 & 0 & 0 & 0 & 0 & 0 & 1 \\
0 & 0 & 0 & 0 & 0 & 1 & 0
\end{array}\right)$};
\node at (2.5,-3) {$+$};
\node at (5,-1.8) {$M'_4$};
\node at (5,-3) {$\left(\begin{array}{ccccccc}
0 & 0 & 0 & 0& 0 & 0 & 0 \\
0 & 0 & 0 & 0 & 0 & 0 & 0 \\
0 & 0 & 0 & 1 & 0 & 0 & 0 \\
0 & 0 & 0 & 0 & 0 & 0 & 1
\end{array}\right)$};
\node at (7.5,-3) {$+$};
\node at (10,-1.8) {$H'$};
\node at (10,-3) {$\left(\begin{array}{ccccccc}
0 & 0 & 0 & 1& 1 & 1 & 0 \\
0 & 0 & 1 & 1 & 1 & 0 & 0\\
0 & 1 & 1 & 0 & 0 & 1 & 2 \\
1 & 1 & 1 & 0 & 0 & 1 & 1
\end{array}\right)$};

\draw[lightgray] (4.9,-0.8)--(6.9,1);

\draw[lightgray] (10.4,-0.8)--(11.9,0.5);
\draw[lightgray] (9.8,0.5)--(10.2,0.9);

\draw[lightgray] (0.9,-3.8)--(1.7,-3.1);
\draw[lightgray] (0,-2.8)--(0.7,-2.1);

\draw[lightgray] (6.4,-3.9)--(7,-3.4);
\draw[lightgray] (4.9,-3.4)--(6.3,-2.1);

\end{tikzpicture}

\caption{An illustration of the matchings $M'_i$ defined in the proof of Lemma \ref{lemreduc} for $n=7$ and $d=(4,6,9,9)$ depicted by their bipartite adjacency matrices.}
\end{figure}\label{fig:lemreduc}

Let $\Delta'(X)$ be the maximum degree in $H'$  of the vertices in $X$. If  $\Delta'(X)>n-k$ then there is a matching $M_{\Delta'(X)}'$ in $H'$ from the vertices of maximum degree in $X$ to $Z$. Color the edges of this matching with $\Delta'(X)$.
By removing this matching from $H'$ we obtain a bipartite multigraph in which the maximum degree of vertices in $X$ is $ \Delta'(X)-1$. By iterating this process we eventually reach a bipartite multigraph $H''$ in which every vertex in $X$ has degree $n-k$  while the maximum degree of the vertices in $Z$ still satisfies $\Delta''(Z)\le n-k$. By K\"onig's theorem, the edge--chromatic number of $H''$ is $n-k$. By combining an  edge--coloring of $H''$ with $n-k$ colors with the  ones obtained in the process of reducing the maximum degree of $H_1$,  the multigraph  $H'$ can be properly edge--colored in such a way that vertex $x_i$ is incident in $H'$ with colors $\{1,\ldots ,d'_i\}$, $1\le i\le k$. We finally obtain a sequential coloring of $H$ by adding to $H'$  the matchings $M_1,\ldots ,M_k$, where $M_j$ gets color $\Delta'(X)+(k-j)$: the largest color of an edge incident to $x_i$ is $d_{H'}(x_i)+t(i)=d_i$. By Proposition \ref{prop:goodascending}, $G$ has a starforest-ASD.
\end{proof}

The next result is a reformulation of Lemma \ref{lem:starforest} which shows that, if a reduced bipartite graph can be decomposed into $n$ star forests with sizes $1,2,\ldots ,n$, regardless of the fact that it is ascending, then the degree sequence of the graph satisfies a condition which is necessary for the existence of a star--forest--ASD.

\begin{lemma} \label{lem-nec}
Let $G=(X\cup Y, E)$ be a reduced bipartite graph with degree sequence $(d_1\le \cdots \le d_k)$.
If $G$ admits an edge--decomposition  into stars forests with centres in $X$,
$$
G=F_1\oplus F_{2}\oplus\cdots \oplus F_n,
$$
where $F_i$ has $i$ edges, then
\begin{equation}\label{eq:good}
\sum_{i=0}^{t-1}d_{k-i}\ge \sum_{i=0}^{t-1}(n-i)\; \mbox{ for each }\; t=1,\ldots, k.
\end{equation}
\end{lemma}

\begin{proof} Let $X=\{x_1,\ldots ,x_k\}$ and $Y=\{y_1,\ldots ,y_{d_k}\}$ be the bipartition of $G$ where $x_i$ is adjacent to $y_1,\ldots ,y_{d_i}$ for each $i$. Since the graph is reduced, the last $d_k-d_{k-1}$ vertices of $Y$ have degree one, the preceding $d_{k-1}-d_{k-2}$ have degree $2$ and, in general,  the consecutive $d_{j}-d_{j-1}$ vertices $y_{d_{j-1}+1},\ldots  ,y_{d_j}$ have degree $j-1$.

 Since $F_n$ has $n$ leaves in $Y$ we clearly have $|Y|=d_k\ge n$. Thus \eqref{eq:good} is satisfied for $t=1$.

Assume that \eqref{eq:good} is satisfied for some $t=j-1<k$.  Consider the subgraph $G_j$ of $G$ induced by $F_n\oplus F_{n-1}\oplus\cdots \oplus F_{n-(j-1)}$.  Since each  vertex in $Y$ has degree at most one in each forest, it has degree at most $j$ in $G_j$. By combining this remark with  the former upper bound on the degrees of vertices in $Y$ we have
\begin{align*}
\sum_{i=0}^{j-1} (n-i)&=\sum_{i=1}^{d_k} d_{G_j}(y_i)\\
&\le \sum_{i=1}^{d_{j-1}} d_{G_j}(y_i)+\sum_{i=d_{j-1}+1}^{d_k} d_G(y_i)\\
&\le jd_{j-1}+(j-1)(d_{j}-d_{j-1})+\cdots +2(d_{k-1}-d_{k-2})+(d_k-d_{k-1})\\
&=\sum_{i=0}^{j-1}d_{k-i},
\end{align*}
and \eqref{eq:good} is satisfied for $t=j$. This concludes the proof.
\end{proof}


\section{An extension Lemma}\label{sec3}

 In this section we prove an extension Lemma which shows that, if $G_R$ admits a star--forest--ASD then so does $G$. This reduces the problem of giving sufficient conditions on the degree sequence of one stable set to ensure the existence of a star--forest--ASD to bipartite reduced graphs. For the proof of our extension Lemma we use the following result by H\"aggkvist \cite{hagg} on edge list--colorings of bipartite multigraphs.

\begin{theorem}\cite{hagg}\label{thm:hagg}
Let $H$ be a bipartite multigraph with stable sets $A$ and $B$. If $H$ admits a sequential coloring, then $H$ can be properly edge--colored for an arbitrary assignment of lists $\{L(a):\; a\in
A\}$ such that $|L(a)|=d(a)$ for each $a\in A$.
\end{theorem}

\begin{lemma}[Extension Lemma]\label{lemexten}
Let $G$ be a bipartite graph with bipartition
$A=\{a_1,\ldots ,a_k\}$ and $B$ and degree
sequence  $d=(d_1\ge \cdots \ge d_k)$, $d=d(a_i)$, of the vertices
in $A$. If the reduced graph $G_R$ of $G$ admits a decomposition
$$
G_R=F'_1\oplus\cdots \oplus F'_t,$$
where each $F'_i$ is a star
forest, then $G$ has an edge decomposition
$$G=F_1\oplus\cdots \oplus F_t$$
where $F_i\cong F'_i$ for each $i=1,\ldots ,t$.
\end{lemma}

\begin{proof}
Let $C$ be the $(k\times t)$ matrix whose entry $c_{ij}$
is the number of edges incident to $a_i$ in the star forest $F'_j$
of the edge decomposition of $G_R$.

Consider the bipartite multigraph $H$ with stable sets $A$ and  $U=\{u_1,\ldots ,u_t\}$,
where $a_i$ is joined with $u_j$ with $c_{ij}$ parallel edges. Now,
for each pair $(i,j)$,  color the $c_{ij}$ parallel edges of $H$
with the neighbors of $a_i$ in the forest $F'_j$  bijectively. Note
that in this way we get a proper edge--coloring of $H$: two edges
incident with a vertex $a_i$ receive different colors since the
 bipartite graph $G_R$ has no multiple edges, and two edges
incident to a vertex $u_j$ receive different colors since $F'_j$ is
a star forest.

By the definition of the bipartite graph $G_R$, each vertex $a_i\in A$ is incident
in the bipartite multigraph $H$ with edges colored  $1,2,\ldots, d_i$. Let $L(a_i)$ be the
list of neighbours of $a_i$ in the original bipartite graph $G$. By Theorem
\ref{thm:hagg}, there is a proper edge--coloring $\chi'$ of $H$ in which the
edges incident to  vertex $a_i$ in $A$ receive the colors from the
list $L(a_i)$ for each $i=1,\ldots ,k$. Now
construct $F_s$ by letting the edge $a_ib_j$ be in $F_s$ whenever
the edge $a_iu_s$ is colored $b_j$ in the latter edge--coloring of
$H$. Thus $F_s$ has the same number of edges than $F'_s$ and the
degree of $a_i$ in $F_s$ is $c_{is}$, the same as in $F_s'$.
Moreover, since the coloring is proper, $F'_s$ is a star forest and two forests $F'_s, F'_{s'}$ are edge--disjoint whenever $s\neq s'$.
This concludes the proof.\end{proof}

We are now ready to prove our main result.

{\it Proof of Theorem \ref{thm:starforest}}: Let $G$ be a bipartite graph with degree sequence $d=(d_1\ge \cdots \ge d_k)$ satisfying $\sum_id_i={n+1\choose 2}$ and $d_i\ge n-i+1$, $i=1,\ldots ,k$. By Lemma \ref{lemasc} there is a $d$--ascending matrix $C$ such that $c_{ij}\ge 1$ for each $(i,j)$ with $i+j\le k+1$. By Lemma \ref{lemreduc} there is a star--forest--ASD of the reduced graph $G_R$ with degree sequence $d$. By Lemma \ref{lemexten} there is also a star--forest--ASD for $G$.\qed




\end{document}